\documentclass[11pt,reqno]{amsart}

\usepackage[dvipsnames]{xcolor}
\usepackage{amsmath}
\usepackage{amsthm}
\usepackage{graphicx}
\usepackage{upgreek}
\usepackage{hyperref}
\usepackage{mathrsfs}
\usepackage{bm,bbm}
\usepackage{amsfonts}
\usepackage{amssymb}
\usepackage{csquotes}
\usepackage{stmaryrd}
\usepackage{ulem}
\usepackage{tikz-cd}
\usepackage{cancel}
\usepackage{siunitx}
\usepackage{algorithm}
\usepackage{algpseudocode}
\usepackage{caption}
\usepackage{subcaption}
\usepackage{booktabs}
\usepackage{fancyhdr}

\usepackage{enumerate}
\usepackage{soul}

\usepackage[margin=1.1in]{geometry}
\numberwithin{equation}{section}

\theoremstyle{plain}
\newtheorem{theorem}{Theorem}[section]
\newtheorem{definition}[theorem]{Definition}

\newtheorem{lemma}[theorem]{Lemma}

\newtheorem{remark}[theorem]{Remark}

\newcommand{\thistheoremname}{}
\newtheorem{genericthm}[theorem]{\thistheoremname}

\newtheorem*{genericthm*}{\thistheoremname}
\newenvironment{assumption*}[1]
  {\renewcommand{\thistheoremname}{#1}%
   \begin{genericthm*}}
  {\end{genericthm*}}

\newcommand{\R}{\mathbb{R}}

\newcommand{\sV}{\mathcal{V}}

\newcommand{\sG}{\mathcal{G}}

\newcommand{\sC}{\mathcal{C}}

\newcommand{\sL}{\mathcal{L}}
\newcommand{\N}{\mathbb{N}}

\newcommand{\intr}{\text{Int}}

\title{Convergence Guarantees for Neural Network-Based Hamilton--Jacobi Reachability}

\author{William Hofgard}
\address{Department of Electrical Engineering, Stanford University, Stanford, CA}
\email{whofgard@stanford.edu}

\def\namedlabel#1#2{\begingroup
    #2%
    \def\@currentlabel{#2}%
    \phantomsection\label{#1}\endgroup
}

\date{\today}

\begin{document}
\maketitle
\begin{abstract}
\noindent
We provide a novel uniform convergence guarantee for DeepReach, a deep learning-based method for solving Hamilton--Jacobi--Isaacs (HJI) equations associated with reachability analysis. Specifically, we show that the DeepReach algorithm, as introduced by Bansal et al. in their eponymous paper from 2020, is \textit{stable} in the sense that if the loss functional for the algorithm converges to zero, then the resulting neural network approximation converges uniformly to the classical solution of the HJI equation, assuming that a classical solution exists. We also provide numerical tests of the algorithm, replicating the experiments provided in the original DeepReach paper and empirically examining the impact that training with a supremum norm loss metric has on approximation error.
\end{abstract}
\tableofcontents
\section{Introduction}
\label{sec:intro}
For general nonlinear optimal control problems, Hamilton--Jacobi (HJ) reachability analysis describes states from which trajectories will eventually reach a set of ``unsafe'' states, even under an optimal control policy \cite{reachability-overview}. The set of states from which trajectories will eventually enter the ``unsafe'' set is typically referred to as the \textit{backward reachable tube} (BRT). Determining the BRT for a given optimal control problem is of particular interest when designing autonomous systems that must avoid collisions or otherwise unsafe configurations (e.g., autonomous vehicles, aircraft, and many other pertinent, real-world examples) \cite{deepreach}.

HJ reachability analysis requires solving the associated Hamilton--Jacobi--Isaacs (HJI) equation, a nonlinear, first-order variational partial differential equation (PDE) \cite{reachability-overview}. Indeed, the HJI equation characterizes the value function for the reachability problem, and the sublevel sets of the value function in turn determine the BRT. In low dimensions, the HJI equation can be solved relatively easily using a grid-based PDE solver. In high dimensions, recent work confronts the so-called ``curse of dimensionality,'' an issue faced by classical, grid-based solvers, by instead using parametrized neural networks to solve high-dimensional PDEs \cite{dgm,deep-bsde-orig}. For instance, Bansal et al. recently developed DeepReach, a deep learning-based solver for the HJI equation that determines approximate BRTs for high-dimensional optimal control problems \cite{deepreach}. However, autonomous systems require robust safety guarantees, and it is often difficult to quantify the approximation error made by neural networks when solving high-dimensional PDEs that lack analytical solutions \cite{deepreach-guarantees}.

We approach this problem from the perspective of PDE theory, aiming to provide a direct uniform convergence guarantee for DeepReach. Building upon the work of \cite{hofgard2024deep, cohen2024deep,l2-pinns}, we show that, under suitable technical assumptions, if the neural network loss for DeepReach approaches zero, then the resulting value function will converge \textit{uniformly} to the true value function, which solves the HJI equation. Empirical evidence for this result is already present in \cite{deepreach-guarantees}, which find that if DeepReach is trained for 100k epochs, then the resulting approximate BRT is close to the true BRT for a baseline example in three dimensions. We also include numerical experiments, replicating the results in \cite{deepreach-guarantees}, focused on multi-vehicle collision avoidance.
\vspace{-0.3cm}
\section{Related Work}
\label{sec:rel-work}

This paper draws upon work done in two areas: deep learning methods for HJ reachability and more general convergence guarantees for deep learning-based PDE solvers. In the latter category, relevant work can be found in \cite{dgm, hofgard2024deep, cohen2024deep, l2-pinns}, which consider solving first-order, nonlinear PDEs using parametrized neural networks. The technical details of establishing convergence may vary with the class of PDEs at hand, but the proof technique presented in \cite{hofgard2024deep} is applicable to DeepReach. Namely, the technique of \cite{hofgard2024deep} can be extended to show that the DeepReach algorithm obtains uniform convergence. Arguing via the properties of viscosity solutions to the HJI equations, which are discussed in detail in \cite{reachability-overview} and \cite{evans-hji}, one may apply a similar technique to DeepReach.

In terms of deep learning methods for HJ reachability, we primarily build upon the work done by Bansal et al. \cite{deepreach}, in which the DeepReach algorithm was first introduced in 2020. As mentioned in Section~\ref{sec:intro}, DeepReach attempts to solve the HJI equation arising is reachability analysis by introducing a parametrized class of deep neural networks that approximately solve the HJI equation. DeepReach's formulation is analogous to the deep Galerkin method (DGM) first presented in \cite{dgm}, which has since become the standard approach for solving generic high-dimensional PDEs. Consequently, much of the uniform convergence analysis for high-dimensional HJB equations, presented in \cite{dgm, hofgard2024deep, cohen2024deep}, applies naturally to DeepReach, with several important technical modifications outlined in Appendix~\ref{sec:proofs}.

Since the publication of DeepReach in 2020 by Tomlin and Bansal, additional analysis of the method's approximation error has been carried out, as in \cite{deepreach-guarantees,conformal-verification,exact-boundary-deepreach,enhancing-deepreach}. For instance, in \cite{deepreach-guarantees}, the authors propose a technique labeled scenario optimization, in which states near the boundary of the approximate BRT are sampled randomly in order to expand the approximate BRT. Consequently, using \cite[Algorithm 1]{deepreach-guarantees}, the authors guarantee that the resulting approximate BRT covers the true BRT with high probability. This probabilistic guarantee, however, requires the implementation of an additional verification algorithm after DeepReach returns an approximate value function, leaving the DeepReach algorithm itself unchanged. In \cite{conformal-verification}, which builds upon the scenario optimization approach of \cite{deepreach-guarantees}, the same authors instead consider conformal prediction, a common technique in establishing confidence bounds in a wide variety of statistical applications. Again, \cite{conformal-verification} ensures that, with high probability, an envelope of the approximate BRT contains the true BRT. This type of analysis is particularly relevant for applications in which the BRT represents unsafe states, such as in collision avoidance.

Finally, several other recent publications modify the training process for DeepReach to improve its approximation error and training efficiency. For instance, \cite{enhancing-deepreach} analyzes the impact that the choice activation has on DeepReach performance, confirming that sinusoidal activation performs better than more standard hyperbolic tangent or ReLU networks. On the other hand, \cite{exact-boundary-deepreach} restructures the DeepReach loss to \textit{guarantee} that DeepReach accurately computes the terminal condition to the HJI equation, before propagating the terminal condition back in time to approximate the rest of the value function. Although \cite{exact-boundary-deepreach} provides no formal convergence guarantees, the restructured training loss typically performs better than the original loss.

\section{Mathematical Background}
\label{sec:prob-statement}

\subsection{HJ Reachability and DeepReach}
We consider the general formulation of HJ reachability presented in \cite{deepreach}. In particular, the dynamics of an autonomous agent in an environment are modeled by a state $x \in \R^n$, a control $u \in \mathcal{U} \subseteq \R^m$, and a disturbance $d \in \mathcal{D} \subseteq \R^m$, with $\dot{x} = f(x, u, d)$. Denote by $\xi^{u, d}_{x, t}(\tau)$ the state at time $\tau$, under the control $u(\cdot)$ and disturbance $d(\cdot)$, starting at time $t$ and state $x$. Denoting the safe set of states by $\sL$, define the BRT on the time interval $[t, T]$, under \textit{worst-case} disturbances $d(\cdot)$, by
\begin{align}
\label{eq:brt}
\sV_{\text{BRT}}(t) := \{x \in \R^n : \forall u(\cdot), \exists d(\cdot), \exists \tau \in [t, T] \text{ such that } \xi^{u,d}_{x,t}(\tau) \in \sL\}.
\end{align}
DeepReach applies just as well to the problems that necessitate backward reach-avoid tubes (BRATs),
in which a set of goal states $\sL$ must be reached by the agent \textit{and} a set of unsafe states $\sG$ must be avoided. In this setting, the BRAT can be formally defined by
\begin{align}
\label{eq:brat}
\sV_{\text{BRAT}}(t) := \{x \in \R^n : \forall d(\cdot), \exists u(\cdot), &\forall s \in [t, T], \xi^{u,d}_{x,t}(s) \notin \sG, \\
&\exists \tau \in [t, T], \xi^{u,d}_{x,t}(\tau), \in \sL\}.
\end{align}
Again, the above setting assumes that the agent acts optimally under the worst-case avoidance, and the BRAT is precisely the set of all states from which the agent can reach the target set $\sL$ after some time while avoiding the unsafe set $\sG$ for all time.

From the above formulation of BRTs and BRATs, we can present the corresponding HJI equations, again based on the discussion from \cite{deepreach}. This paper only considers the case of BRTs for simplicity. Given a target set $\sL$, define a function $\ell : \R^n \to \R$ such that $\sL = \{x \in \R^n : \ell(x) \leq 0\}$. Define a cost functional, starting at initial state $x$ and initial time $t$, by
\begin{align}
\label{eq:cost}
J(x, t, u(\cdot), d(\cdot)) = \min_{\tau \in [t, T]} \ell(\xi^{u,d}_{x,t}(\tau)).
\end{align}
In the case that $\sL$ is an unsafe set that the agent wishes to avoid for all time $t$, the agent first selects an optimal control that maximizes $\ell$, and the disturbance (viewed as an adversarial player in a zero-sum, two-player game) selects the input that minimizes $\ell$, yielding a value function of the form
\begin{align}
\label{eq:value-func}
V(t, x) := \inf_{d(\cdot)} \sup_{u(\cdot)} \left\{J\left(x, t, u(\cdot), d(\cdot)\right)\right\}.
\end{align}
Via a standard dynamic programming argument, presented in \cite{reachability-overview}, one obtains the HJI variational equation for the reachability problem, given by
\begin{align}
\label{eq:hji-reach}
\begin{split}
&\min\left\{\partial_t V(t, x) + H(t, x), \ell(x) - V(t, x)\right\} = 0, \\
& V(T, x) = \ell(x).
\end{split}
\end{align}
Above, the Hamiltonian $H : \R \times \R^n \to \R$ is defined by
\begin{align}
\label{eq:hamiltonian}
H(t, x) := \sup_{u} \inf_{d} \langle \nabla_x V(t, x), f(x, u, d) \rangle.
\end{align}
Roughly speaking, the term $\ell(x) - V(t, x)$ in Equation~\eqref{eq:hji-reach} captures the case in which an optimal trajectory enters the set $\sL$ prior to time $T$, as if $V(t,x) = \ell(x)$, then $\xi^{u, d}_{x,t}(\tau) \in \sL$ for some $\tau \in [t, T]$.
The BRT itself is computed as a sublevel set of the solution $V$ to the HJI equation, with
\begin{align*}
    \sV_{\text{BRT}}(t) := \{x : V(t, x) \leq 0\},
\end{align*}
and a corresponding optimal control
\begin{align*}
    u^\star(t, x) := \arg \max_{u(\cdot)} \min_{d(\cdot)} \langle \nabla_x V(t, x), f(x, u, d) \rangle.
\end{align*}
To solve the Equation~\eqref{eq:hji-reach} numerically, consider a parametrized neural network $V_\theta : \R \times \R^n \to \R$, with parameters $\theta \in \R^P$, that approximates the true solution to Equation~\eqref{eq:hji-reach}. At this point, the architecture of the neural network remains unspecified, although the original implementation of DeepReach utilizes a fully-connected, feedforward neural network with sinusoidal activations between layers \cite{deepreach}. This is the same architecture considered in subsequent papers that extend DeepReach, as in \cite{deepreach-guarantees} and \cite{enhancing-deepreach}. In \cite{deepreach}, a standard $L^2$-loss functional is used. Below, we make a slight modification, instead using \textit{sup-norm loss} for both our theoretical guarantees an numerical experiments:
\begin{align}
\label{eq:nn-loss}
\begin{split}
&h_1(\theta) := \|V_\theta(T, x) - \ell(x)\|_\infty, \\
&h_2(\theta) := \|\min\{\partial_t V_\theta(t, x) + H_\theta(t,x), \ell(x) - V_\theta(t, x)\}\|_{\infty}, \\
&L(\theta) := h_1(\theta) + \lambda h_2(\theta),  
\end{split}
\end{align}
where, in practice, $(t, x) \in \R \times \R^n$ are sampled points, and $H_\theta$ is the Hamiltonian associated with the approximate value function $V_\theta$, defined as in Equation~\eqref{eq:hamiltonian}. The DeepReach training procedure then mimics stochastic gradient descent (SGD). In particular, the neural network is first trained with $\lambda = 0$ to appropriately learn the terminal condition of the PDE, at $t = T$. Then, $t$ is linearly decreased from $t = T$ to $t = 0$, sampling $K$ points $\{x_k\}_{k=1}^K$ at each step and performing an SGD step on the neural network parameters $\theta$ to learn the solution to Equation~\eqref{eq:hji-reach}. 
\begin{remark}
\normalfont In \cite{deepreach}, the above loss is defined with respect to the $\ell^1$-norm or $\ell^2$-norm. However, recent research surrounding neural network solutions to second-order HJB equations indicates that such a choice of loss may not result in convergence to the true value function \cite{l2-pinns}. Instead, using a sup-norm loss metric as above brings both theoretical benefits, as seen in the convergence proof in Appendix~\ref{sec:proofs}, and practical benefits, as seen in Section~\ref{sec:experiments} below. As noted above, the sup-norms involved in Equation~\eqref{eq:nn-loss} are approximated by sampling in practice: \cite{hofgard2024deep} and \cite{cohen2024deep} discuss convergence guarantees for this sampling-based approach. For our theoretical guarantees, we assume that the sup-norm loss in Equation~\eqref{eq:nn-loss} can be approximated to arbitrary accuracy.
\end{remark}
\subsection{Baseline Collision Avoidance Application}
As a baseline example, consider the case of two airplanes approaching each other midair, as in the original DeepReach paper \cite{deepreach}. In the worst-case scenario, one of the planes acts as a pursuer that attempts to collide with the other plane, while the other plane acts as an evader. The relative dynamics of this system, presented in \cite{deepreach} and \cite{exact-boundary-deepreach}, are given by
\begin{align}
\label{eq:plane-dynamics}
& \dot{x}_1 = -v_e + v_p \cos \theta + \omega_e x_2, \quad \dot{x}_2 = v_p \sin \theta - \omega_e x_1, \quad \dot{\theta} = \omega_p - \omega_e, 
\end{align}
where $x = (x_1, x_2)$ represents the relative positions of the two planes in a two-dimensional plane, and $\theta$ is their relative heading. The two velocities $v_e$ and $v_p$ are the evader's and pursuer's velocities respectively, and the angular velocities $\omega_e$ and $\omega_p$ are defined analogously. For simplicity, the velocities $v_e$ and $v_p$ are held constant, while the two agents select inputs $\omega_e$ and $\omega_p$. Note that, in our notation above, $\omega_p$, the control input for the pursuer, is actually the \textit{disturbance} in the context of the reachability problem. We constrain $|\omega_e|, |\omega_p| \leq \omega_{\max}$ for some constant $\omega_{\max} > 0$, and consider a standard unsafe set of the form $\sL = \{x : \|x\|_2 \leq \beta\}.$

Taking $\ell(x) = \|x\|_2 - \beta$, where $\beta > 0$ is a safety parameter, it follows that $\sL = \{x : \ell(x) \leq 0\}$, as is desired in the setting of HJ reachability. Consequently, we can carry out the DeepReach algorithm as outlined above, training the neural network using the DeepReach loss in Equation~\eqref{eq:nn-loss} to obtain an approximation $V_\theta$ for the value function. In turn, the sublevel set
\begin{align*}
\sV^\theta_{\text{BRT}}(t) := \{x : V_\theta(t, x) \leq 0\}
\end{align*}
will approximate the true BRT for the collision avoidance problem. As noted in \cite{analytic-brt}, this example is a natural baseline to consider for reachability analysis because the \textit{true} BRT $\sV$ is possible to describe analytically (at least, very nearly, up to some sampling error). Additionally, as shown in \cite{deepreach}, it is possible to analytically compute the Hamiltonian for this problem. Denoting
\begin{align*}
  \begin{bmatrix}p_1 \\ p_2 \\ p_3 \end{bmatrix} := \begin{bmatrix} \partial_{x_1} V(t, x_1, x_2, \theta) \\ \partial_{x_2} V(t, x_1, x_2, \theta) \\ \partial_{\theta} V(t, x_1, x_2, \theta) \end{bmatrix},
\end{align*}
we have that
\begin{align}\label{eq:analytical-hamiltonian}
  H(t, x) = p_1(-v_e + v_p \cos x_3) + p_2(v_p \sin x_3) - \omega_{\max}|p_1 x_3 - p_2 x_1 - p_3| + \omega_{\max} p_3
\end{align}
Because the Hamiltonian can be computed explicitly (rather than solving an optimization problem), implementing the DeepReach loss is straightforward in this instance. For general reachability problems, however, it is possible to efficiently solve the optimization problem in Equation~\eqref{eq:hamiltonian} at each step \cite{deepreach}.

\section{Theoretical Guarantees}
\label{sec:approach}
We provide two results concerning the convergence of DeepReach for solving the HJI equation in Equation~\eqref{eq:hji-reach}. All technical results are presented in full, including detailed proofs, in Appendix~\ref{sec:proofs}. First, we obtain an \textit{existence} result as a relatively straightforward consequence of the universal approximation power of neural networks. Specifically, neural networks with smooth, bounded, non-constant activations (as is the case in the sinusoidal implementation of DeepReach) can approximate functions and their derivatives \textit{arbitrarily} well, as in \cite[Theorem 3]{univ-approx-thm}. Such universal approximation results are non-constructive in nature, but they guarantee that a sequence of neural networks taking the loss in Equation~\eqref{eq:nn-loss} to zero exists, the first step in obtaining a uniform convergence guarantee. Under the assumption that the value function $V$ that solves Equation~\eqref{eq:hji-reach} is locally-Lipschitz with locally-Lipschitz gradients, denoted by $V \in \sC^{1,1}_{\text{loc}}([0, T] \times \R^n)$, we prove the following:
\begin{theorem}
\label{th:existence}
For every $\varepsilon > 0$, there exists a constant $C > 0$ that depends only on the dynamics $\dot{x} = f(x, u, d)$ of the underlying reachability problem such that for some $\theta \in \R^P$, the DeepReach loss in Equation~\eqref{eq:nn-loss} satisfies $L(\theta) \leq C\varepsilon$.
\end{theorem}
As remarked above, the proof of this theorem involves an application of the universal approximation theorem in \cite[Theorem 3]{univ-approx-thm} and subsequent algebraic manipulations of the DeepReach loss. An existence result, while encouraging, is not immediately useful. However, building upon Theorem~\ref{th:existence}, we will also prove the following \textit{convergence} result, which is far more practical in nature and does not rely upon a non-constructive universal approximation theorem.
\begin{theorem}
\label{th:convergence}
If a sequence of parameters $\{\theta^{(k)}\}_{k \in \N}$ is such that $L(\theta^{(k)}) \to 0$ as $k \to \infty$, then $V_\theta \to V$ uniformly, in the sense that for any compact set $K \subset \R^n$,$$\sup_{(t,x) \in [0,T] \times K} |V_{\theta^{(k)}}(t,x) - V(t, x)| \to 0$$as $k \to \infty$.
\end{theorem}
The proof of Theorem~\ref{th:convergence} relies on tools from the theory of viscosity solutions, a notion of weak solutions to HJI equations (and more generally, first-order nonlinear PDEs common in optimal control) first introduced in \cite{evans-hji}. Specifically, by establishing that a sequence of neural network approximators are viscosity solutions to a sequence of perturbed HJI equations and applying an appropriate comparison principle to two suitably-defined upper and lower limits of the neural network approximators, one may obtain the uniform convergence guarantee in Theorem~\ref{th:convergence}. This proof technique is presented rigorously in Appendix~\ref{sec:proofs}, and it follows the approach of \cite{hofgard2024deep} with several important deviations. Given the structure of Equation~\eqref{eq:hji-reach}, this technique generalizes well to the context of HJ reachability and DeepReach. 
\begin{remark}
\normalfont Note that Theorem~\ref{th:convergence} is agnostic towards the choice of optimization algorithm for minimizing the DeepReach loss. In particular, the convergence guarantee provided by the above result \textit{still} holds if a gradient-free (or even randomized) algorithm or heuristic is used to optimize the DeepReach loss. Although stochastic gradient descent is the standard method for optimizing the class of neural networks considered in \cite{deepreach,deepreach-guarantees,exact-boundary-deepreach}, the above convergence guarantee holds in any optimization setting.
\end{remark}
\section{Numerical Experiments}
\label{sec:experiments}
In this section, we provide numerical experiments that demonstrate the performance of DeepReach for the collision avoidance problem in Section~\ref{sec:prob-statement}. The first problem, presented in \cite{analytic-brt} and solved using DeepReach in \cite{deepreach, deepreach-guarantees}, is a standard three-dimensional reachability problem for testing neural network-based PDE solvers. All relevant code can be found \href{https://colab.research.google.com/drive/18MN5imhNE1xZs9bR64pnKOFFu7z9l0ic?usp=sharing}{this notebook}, which carries out experiments for the two-vehicle collision avoidance example from Section~\ref{sec:prob-statement}.

At each step of DeepReach training, $K = 65000$ uniformly-sampled points are used to train the neural network as in Equation~\eqref{eq:nn-loss}. All states are scaled to lie in the interval $[-1, 1]$, and the time interval in question is rescaled to satisfy $T = 1$, both of which empirically improve training stability. As in \cite{deepreach}, we utilize a simple three-layer feedforward neural network, with hidden layer size 512 and sinusoidal activation given by $\sigma(x) = \sin(x)$. The hyperparameter $\lambda$ is set to $\lambda = 0$ during pre-training and $\lambda = 1.5 \times 10^2$ during training, and the neural network loss in minimized using the Adam optimizer. The true BRT for the collision avoidance problem is computed analytically as in \cite{analytic-brt} and numerically, using the Level Set Toolbox (LST) PDE solver, using the implementation provided by \cite{deepreach}. LST also provides a numerical solution for the true value function $V$. 

Our implementation builds upon the open-source PyTorch implementation of the algorithm provided by the authors of \cite{deepreach}. In particular, we utilize the problem setup, neural network architecture, and training procedure outlined from the \href{https://github.com/smlbansal/deepreach/tree/master}{open-source package} provided by \cite{deepreach}. These results are displayed in Figures~\ref{fig:valfunc-diff}--~\ref{fig:brt-comp-pretrained}. On the other hand, Figures~\ref{fig:brt-comp-finetuned}--~\ref{fig:max-error} utilize pre-trained models from \cite{deepreach} that are fine-tuned with sup-norm loss instead of the $\ell^1$-loss metric considered in \cite{deepreach}.

In more detail, Figure~\ref{fig:valfunc-diff} illustrates the absolute difference between four slices of the true and approximate value functions, evaluated at time $t = 0.7$ and on the box $(x_1, x_2) \in [-1,1]^2$. Note that the figures should be interpreted in the setting of the collision avoidance problem introduced in Section~\ref{sec:prob-statement}, in which $(x_1, x_2, \theta)$ are the \textit{relative} coordinates between two agents. For $\theta = \pi$, which represents the two agents facing towards each other in relative coordinates, the pre-trained DeepReach approximation deteriorates in quality. This phenomenon is also reflected in Figure~\ref{fig:brt-comp}, which compares the approximate BRT computed via DeepReach to the true BRT, computed analytically as in \cite{analytic-brt}. However, after fine-tuning the model with sup-norm loss as defined in Equation~\eqref{eq:nn-loss}, the DeepReach BRT more closely approximates BRT for $\theta = \pi$, as seen in Figure~\ref{fig:brt-comp-pretrained} and Figure~\ref{fig:brt-comp-finetuned}. After only 1K epochs (approximately 5 minutes on an L4 Tensor Core GPU) of training, this fine-tuning procedure exhibits noticeable improvements over models pre-trained for 100K epochs (approximately 16 hours on an L4 GPU) with $\ell^1$-loss, both in terms of the loss metric itself \textit{and} the maximum approximation error, as seen in Figure~\ref{fig:deepreach-loss} and Figure~\ref{fig:max-error} respectively.\begin{figure}[htpb]
\begin{center}
\centerline{\includegraphics[width=\linewidth]{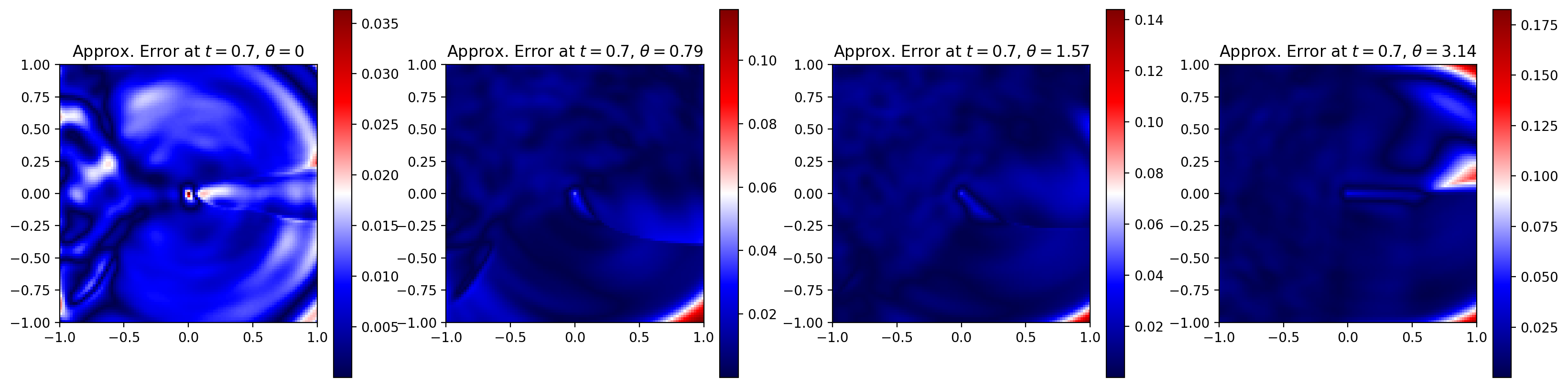}}
\caption{Absolute difference between true and approximate value functions, using pre-trained DeepReach model from \cite{deepreach}.}
\label{fig:valfunc-diff}
\end{center}
\end{figure}
\begin{figure}[htpb]
\begin{center}
\centerline{\includegraphics[width=1\linewidth]{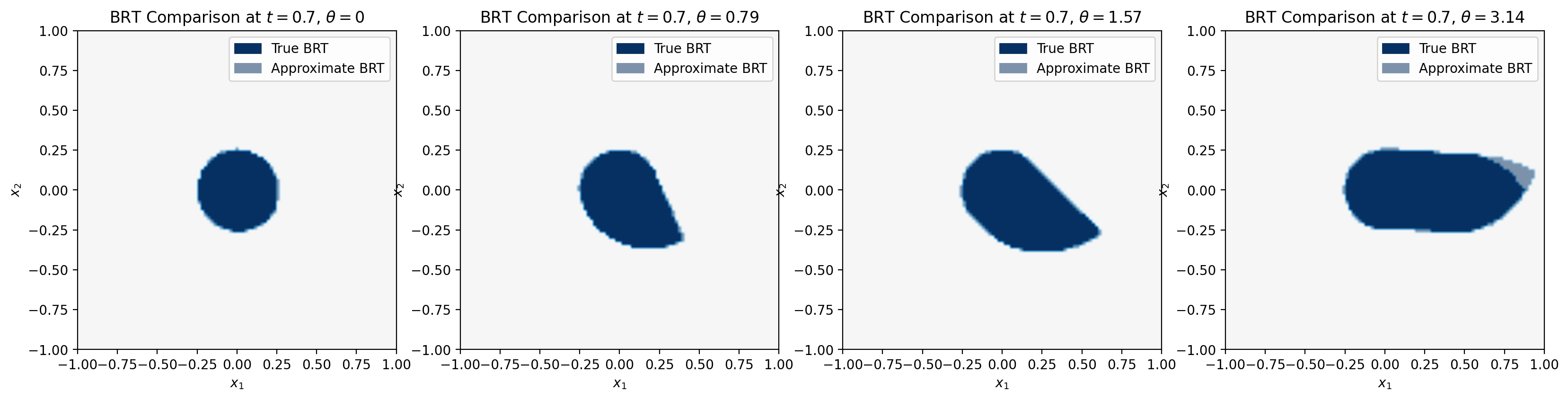}}
\caption{Comparison of BRT obtained by pre-trained DeepReach model and analytical BRT from \cite{analytic-brt}. Note the discrepancy between the two sets at $\theta = \pi$. }
\label{fig:brt-comp}
\end{center}
\end{figure}

\begin{figure}[htpb]
  \centering
	\begin{minipage}[b]{0.48\linewidth}
		\centering
		\includegraphics[width=\textwidth]{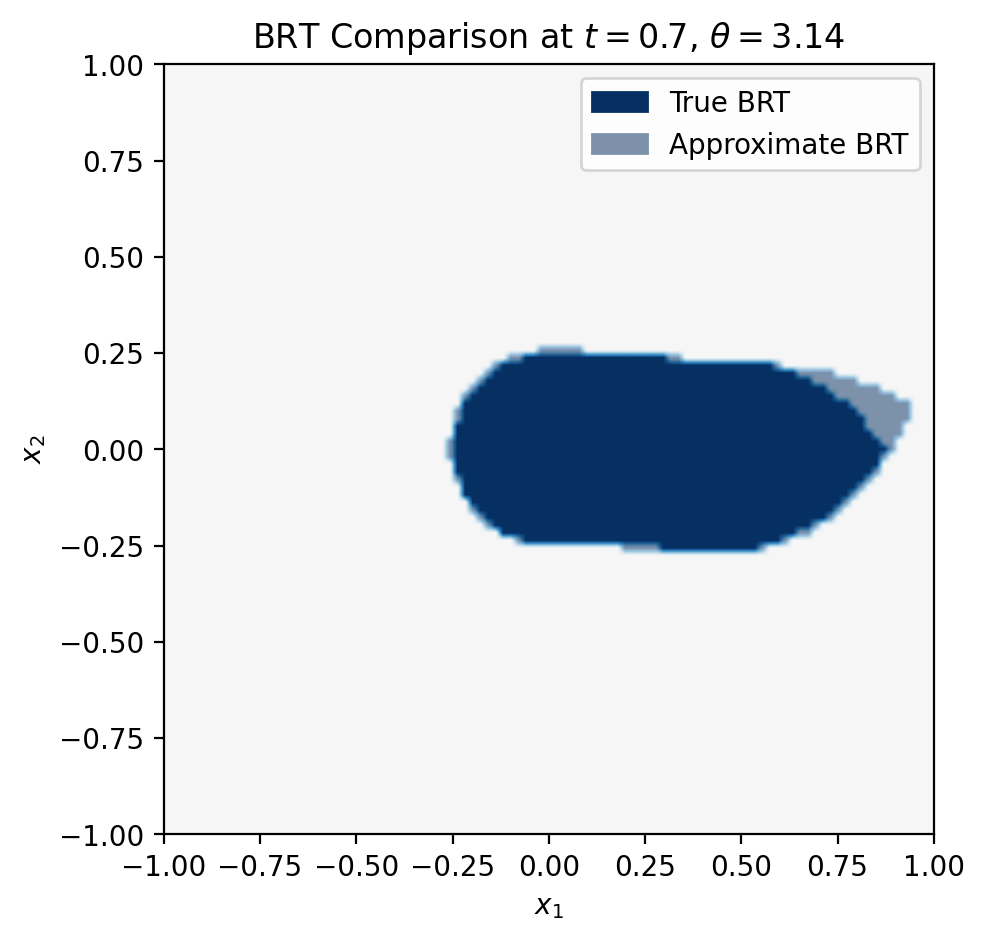}
		\caption{BRT comparison, using pre-trained  model from \cite{deepreach}.}
		\label{fig:brt-comp-pretrained}
	\end{minipage}
	\begin{minipage}[b]{0.48\linewidth}
		\centering
		\includegraphics[width=\textwidth]{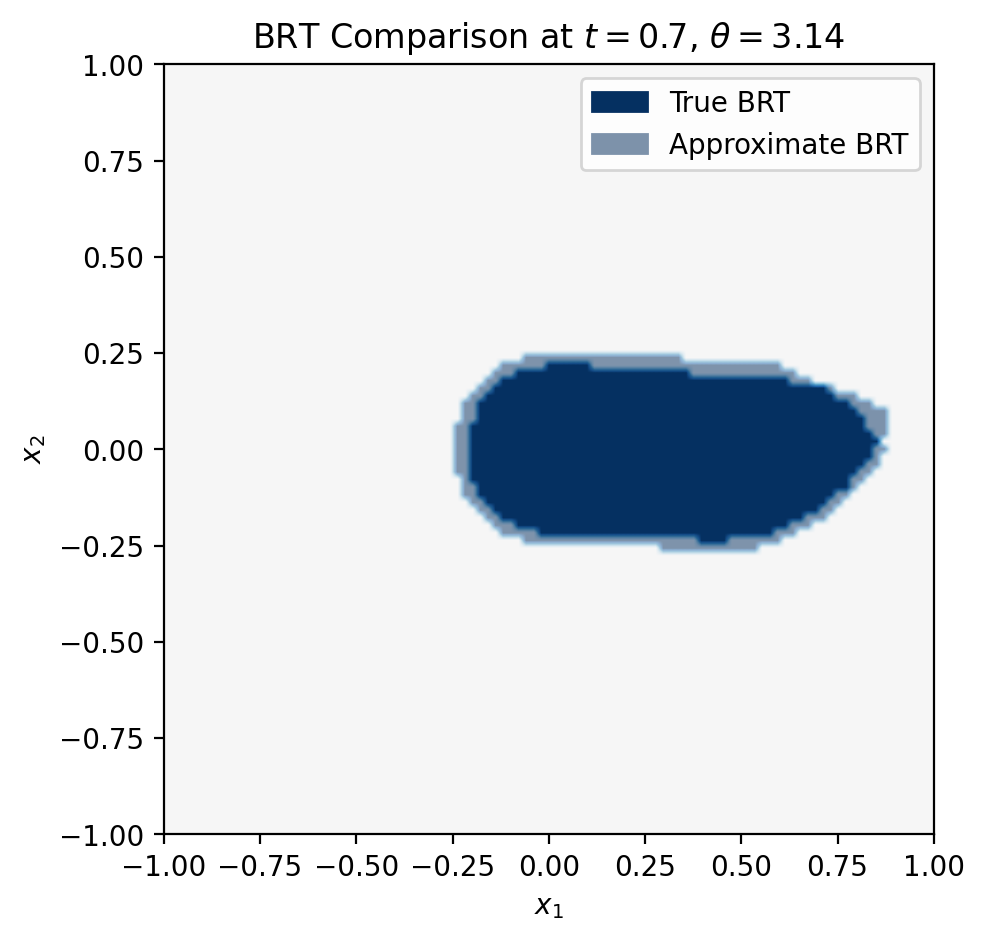}
		\caption{BRT comparison, using fine-tuned model with sup-norm loss.}
		\label{fig:brt-comp-finetuned}
	\end{minipage}
\end{figure}

\begin{figure}[htpb]
  \centering
	\begin{minipage}[b]{0.46\linewidth}
		\centering
		\includegraphics[width=\textwidth]{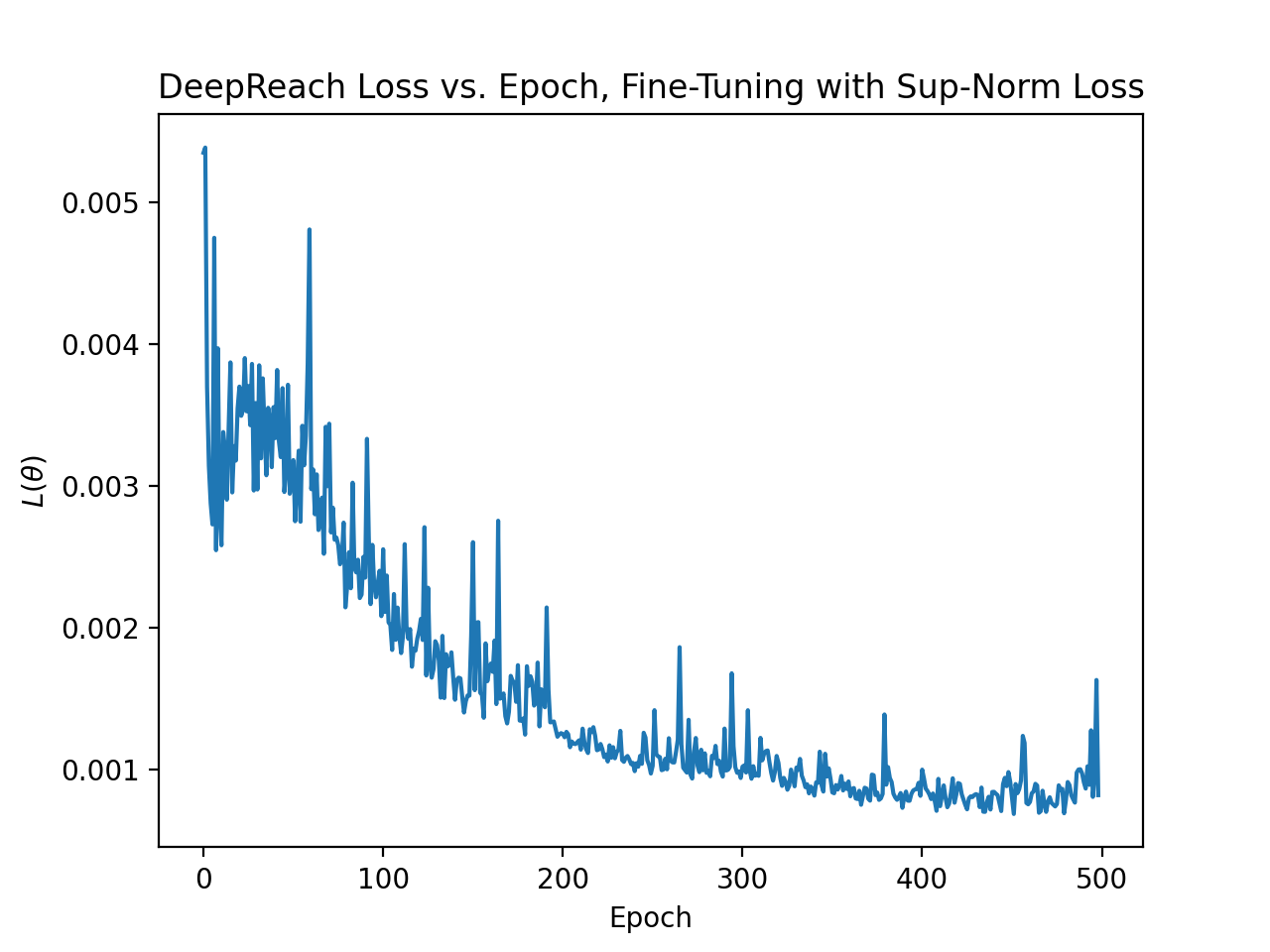}
		\caption{Fine-tuning with sup-norm loss results in significant decreases in DeepReach loss over just $500$ epochs.}
		\label{fig:deepreach-loss}
	\end{minipage}
	\hspace{0.5cm}
	\begin{minipage}[b]{0.46\linewidth}
		\centering
		\includegraphics[width=\textwidth]{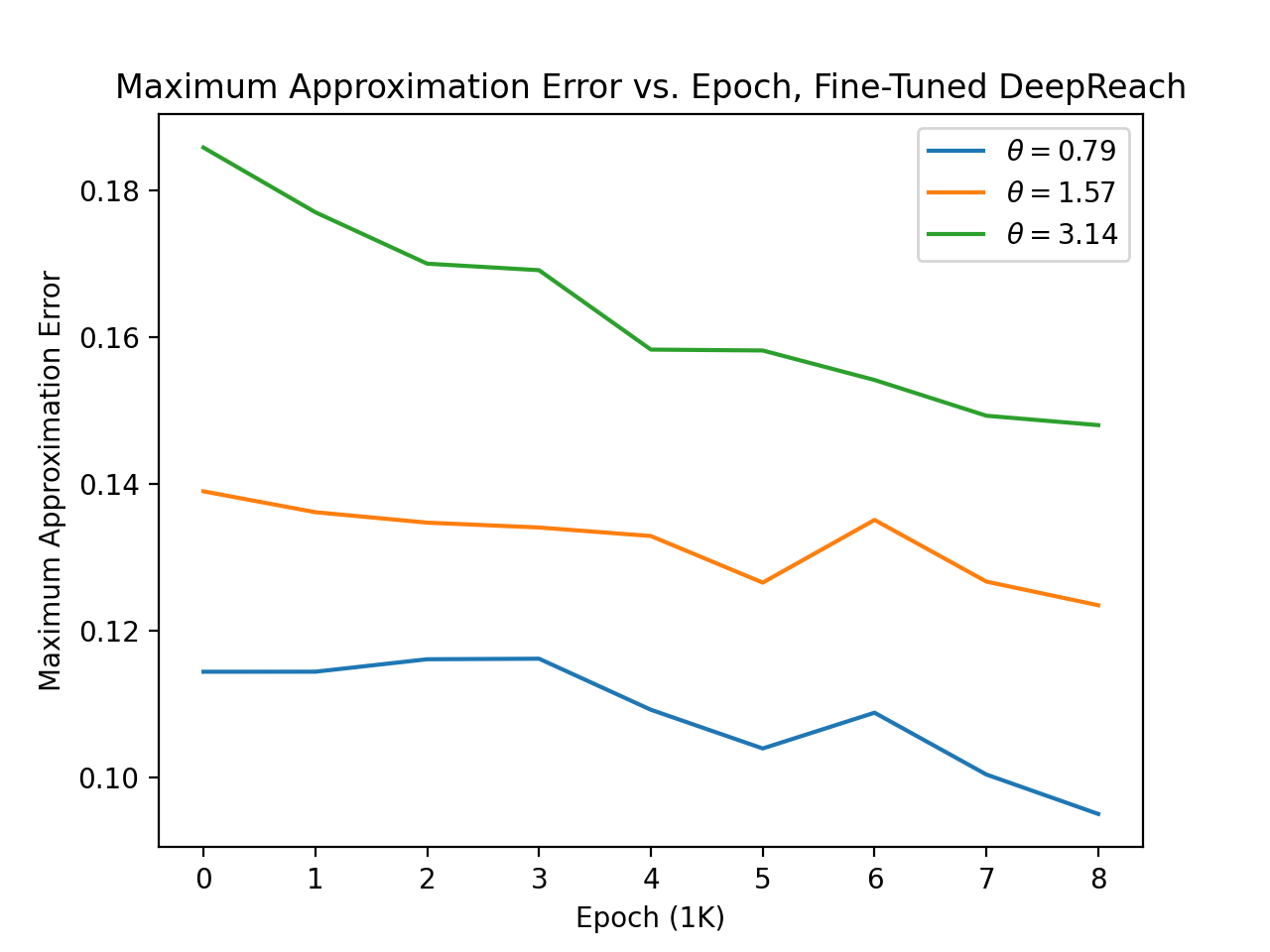}
		\caption{Training becomes slightly less stable over 10K epochs, but maximum approximation error continues to decrease.}
		\label{fig:max-error}
	\end{minipage}
\end{figure}

\section{Conclusions and Future Work}
\label{sec:conclusions}
We provide two novel convergence guarantees for a variant of the DeepReach algorithm, first introduced in \cite{deepreach}. In particular, by usinga  sup-norm loss metric, we show that the DeepReach loss can be made arbitrarily small by neural network approximators (Theorem~\ref{th:existence}) before showing that \textit{any} sequence of neural networks that takes the DeepReach loss to zero must converge uniformly to the true solution to the HJI equation in question (Theorem~\ref{th:convergence}). These two results prove that DeepReach is robust in the sense that is guaranteed to converge uniformly to the true value function that solves Equation~\eqref{eq:hji-reach} as long as the Deepreach loss in Equation~\eqref{eq:nn-loss} converges to zero. Furthermore, we demonstrate this convergence result empirically, showing that by using the sup-norm loss metric, pre-trained DeepReach models can be fine-tuned to a greater degree of accuracy.

Future work may aim to establish an explicit rate of convergence that depends on the DeepReach loss, as in \cite{cohen2024deep}, or investigate whether Theorem~\ref{th:convergence} can be extended to HJI equations that do \textit{not} admit classical solutions. Extending our convergence result to the altered training approach in \cite{exact-boundary-deepreach} is also of interest. In terms of implementation, improving the performance of DeepReach by utilizing \textit{adversarial training}, a common technique for training neural networks with sup-norm loss, is possible \cite{l2-pinns}. Additionally, experimenting with more expressive neural network architectures may yield better approximations.
\newline

\noindent
\textbf{Acknowledgments.} The author would like to thank Prof. Marco Pavone, Dr. Daniele Gammelli, Daniel Morton, and Matthew Foutter for their advice and feedback on this paper.
\bibliographystyle{unsrt}
\bibliography{refs}
\newpage
\appendix
\section{Technical Proofs}\label{sec:proofs}
In this section, we provide proofs of Theorem~\ref{th:existence} and Theorem~\ref{th:convergence} respectively. First, some additional notation is necessary. Consider an instantiation of Equation~\eqref{eq:hji-reach} on $\R \times \R^n$. Elements of the relevant class of fully-connected neural networks, as considered in \cite{deepreach,exact-boundary-deepreach} take the form
a network with $L$ layers, maximum width $n$, and a common activation function $\sigma$ take the form
\begin{align}
\label{eq:multi-layer-nn}
    V_\theta(t, x) := \sigma(W_{L} \ldots \sigma(W_1 x + \alpha t + c_1) \ldots + c_{L}),
\end{align}
where the activation function $\sigma$ is applied elementwise. Above, $W_i$ are weight matrices, $c_i$ are bias vectors, and $\alpha$ is a scalar weight. In turn, the parameters of each neural network are of the form $\theta = (W_1, \ldots, W_L, c_1, \ldots, c_L, \alpha) \in \R^P$ (upon flattening all weight matrices into vectors), where $P$ depends on the maximum width of the network, the depth $L$ of the network, and the dimension $n$ of the target HJI equation. In turn, we take $\mathfrak{C}^{(P)}_{n+1}(\sigma)$ to be the class of neural networks with parameters $\theta$ of dimension at most $P$ (but any number of layers $L$), from which we define $\mathfrak{C}_{n+1}(\sigma)$ as the class of neural networks with parameters $\theta$ of any dimension (i.e., allowing networks with unbounded width).

Additionally, given any compact subset $K \subset \R^n$, define the standard norm on $\sC^{1}(K)$, the space of continuously-differentiable functions $f : K \to \R$, given by
\begin{align*}
  \|f\|_{\sC^{1}(K)} := \sup_{x \in K} |f(x)| + \sup_{x \in K} \|\nabla f (x)\|_2
\end{align*}
Finally, define an operator by
\begin{align*}
  \mathcal{L}[V](t, x) := \min\{\partial_t V_\theta (t, x) + H_\theta(t, x), \ell(x) - V_\theta(t, x)\},
\end{align*}
observing that the first line of Equation~\eqref{eq:hji-reach} becomes $\mathcal{L}[V](t, x) = 0$. Throughout this section, we also impose the following mild technical assumptions.

\begin{assumption*}{Assumption A}
\label{a:assump-A}
We impose the following assumptions on the underlying reachability problem, with dynamics $\dot{x} = f(x, u, d)$:
\begin{enumerate}
\item[(1)]
There exist compact sets $\mathcal{U} \subseteq \R^m$ and $\mathcal{D} \subseteq \R^m$ such that all control inputs $u$ and disturbances $d$ satisfy $u \in \mathcal{U}$ and $d \in \mathcal{D}$.
\item[(2)]
There exists a compact set $K \subseteq \R^n$ such that for any initial state $x(0) \in K$ and sequence of optimal control inputs and disturbances $u \in \mathcal{U}$, $d \in \mathcal{D}$, the state trajectory satisfies $x(t) \in K$ for all $t \in [0, T]$.
\item[(3)] There exists a constant $C_f > 0$ such that for any $x \in K$, $\|f(x, u, d)\|_2 \leq C_f$ for all $(u, d) \in \mathcal{U} \times \mathcal{D}$.
\end{enumerate}
\end{assumption*}
For instance, these standard assumptions all hold for the collision avoidance example introduced in Section~\ref{sec:prob-statement}. By scaling all states to lie in the $d$-dimensional box $[-1, 1]^d$ in Section~\ref{sec:experiments} above, we ensure that the compactness assumption above is always met. If dynamics $f(x, u, d)$ are continuous, then the fact that all trajectories lie in some compact set $K$ immediately implies the boundedness assumption above. Effectively, we require bounded state trajectories, bounded control inputs and disturbances, and bounded dynamics. 

The next assumption, placed on Equation~\eqref{eq:hji-reach}, is discussed in detail in \cite{reachability-overview} and \cite{evans-hji}. Under a wide variety of circumstances, Equation~\eqref{eq:hji-reach} will indeed admit unique solutions.
\begin{assumption*}{Assumption B}
\label{a:assump-B}
There exists a unique, classical solution $V \in \sC^{1,1}(\Omega_K)$ to Equation~\eqref{eq:hji-reach}. Furthermore, $V$ is also the unique viscosity solution to Equation~\eqref{eq:hji-reach}, as defined in Definition~\ref{def:visc-solutions} below.
\end{assumption*}
At this point, we note that certain HJI equations \textit{only} admit viscosity solutions, the appropriate definition of a weak solution for first-order, nonlinear PDEs that resemble the HJI equation. The framework for establishing neural network approximation and convergence guarantees for equations that admit viscosity solutions \cite{hofgard2024deep}, however, does not yet extend to the case of equations that do not admit unique classical solutions. In many cases, however, including (again) the collision avoidance example in Section~\ref{sec:prob-statement}, this is not a concern.

With this notation in mind, we first prove Theorem~\ref{th:existence}, relying on the approximation guarantees provided by \cite[Theorem 3]{univ-approx-thm}. To make the proof more palatable, we construct several technical lemmas.
\begin{lemma}\label{lemma:univ-approx-est}
Let $V \in \mathcal{C}^{1,1}(\Omega_K)$ solve Equation~\eqref{eq:hji-reach} on $[0, T] \times K$, for some compact set $K \subset \R^n$. Then, for any $\varepsilon > 0$ and bounded, non-constant activation $\sigma : \R \to \R$, there exists $V_\theta \in \mathfrak{C}_{n+1}(\sigma)$ such that
\begin{align}
  \label{eq:univ-approx-est}
  \begin{split}
  \varepsilon > \sup_{(t,x) \in \Omega_K} |V(t,x) - V_\theta(t, x)| &+ \sup_{(t,x) \in \Omega_K} |\partial_t V(t,x) - \partial_t V_\theta(t, x)| \\
  &+ \sup_{(t, x) \in \Omega_K} \|\nabla_x V(t,x) - \nabla_x V_\theta(t, x)\|_2.
  \end{split}
\end{align}
\end{lemma}
\begin{proof}
This is a direct consequence of the universal approximation theorem for neural networks, stated in Proposition~\cite[Theorem 3]{univ-approx-thm}. In particular, the result therein states that if $\sigma \in \sC^m(\R)$ is a non-constant and bounded activation, then $\mathfrak{C}_{n+1}$ is uniformly $m$-dense on compact sets in $\sC^m(\R^{n+1})$. In particular, for all $h \in \sC^m(\R^{n+1})$, all compact subsets $K \subset \R^{n+1}$, and any $\varepsilon > 0$, there exists $V_\theta \in \mathfrak{C}_{n+1}$ such that $\|h - \psi\|_{\sC^m(K)} < \varepsilon$. Taking $m = 1$ suffices to prove the lemma.
\end{proof}
\noindent
Equipped with the preceding lemma, the following estimates will allow us to prove Theorem~\ref{th:existence}.
\begin{lemma}\label{lemma:unif-approx-bounds}
Assume that $\varepsilon > 0$ and $V_\theta$ are as in Lemma~\ref{lemma:univ-approx-est} and
\begin{align*}
  H_\theta(t, x) := \sup_{u} \inf_d \langle \nabla_x V_\theta(t, x), f(x, u, d)\rangle.
\end{align*}
Then, there exists $C_f > 0$ depending only on the dynamics $f$ such that
\begin{align*}
  |H_\theta(t, x) - H(t, x)| < C_f \varepsilon
\end{align*}
for any $(t, x) \in \Omega_K$. Furthermore, $|V_\theta(T, x) - \ell(x)| < \varepsilon$ for all $x \in K$.
\end{lemma}
\begin{proof}
The second estimate follows easily from Lemma~\ref{lemma:univ-approx-est} and the fact that $V(T, x) = \ell(x)$. Indeed, the fact that $\sup_{(t,x) \in \Omega_K} |V(t,x) - V_\theta(t, x)|  < \varepsilon$ implies that
\begin{align*}
  |\ell(x) - V_\theta(T, x)| = |V(T, x) - V_\theta(T, x)| < \varepsilon
\end{align*}
for all $x \in K$.  On the other hand, Lemma~\ref{lemma:univ-approx-est} also guarantees that for any $(t, x) \in \Omega_K$,
\begin{align*}
  \sup_{(t, x) \in \Omega_K} \|\nabla_x V(t,x) - \nabla_x V_\theta(t, x)\|_2 < \varepsilon.
\end{align*}
As a result, for \textit{fixed} $(t, x) \in \Omega_K$, taking 
\begin{align*}
  p(t, x) := \frac{1}{\varepsilon}(\nabla_x V(t,x) - \nabla_x V_\theta(t, x)) \in \R^n,
\end{align*}
it follows that
\begin{align*}
  \nabla_x V_\theta(t,x) = \nabla_x V(t, x) + \varepsilon p(t, x),
\end{align*}
where $\|p(t, x)\|_2 < 1$. In turn, we can write
\begin{align*}
  |H_\theta(t, x) - H(t, x)| &= \left|\sup_{u} \inf_d \langle \nabla_x V_\theta(t, x), f(x, u, d)\rangle - \sup_{u} \inf_d \langle \nabla_x V(t, x), f(x, u, d)\rangle\right| \\
  &= \left|\sup_{u} \inf_d \langle \nabla_x V(t, x) + \varepsilon p(t, x), f(x, u, d)\rangle - \sup_{u} \inf_d \langle \nabla_x V(t, x), f(x, u, d)\rangle\right| \\
  &=\varepsilon \left| \sup_{u} \inf_d \langle p(t, x), f(x, u, d)\rangle\right| \\
  &\leq \varepsilon \sup_u \left|\inf_d \langle p(t, x), f(x, u, d)\rangle \right| \\
  &\leq \varepsilon \sup_u \sup_d \left|\langle p(t, x), f(x, u, d)\rangle\right| \\
  &\leq \varepsilon \sup_u \sup_d \|p(t, x)\|_2 \|f(x, u ,d)\|_2 \\
  &< \varepsilon \sup_u \sup_d \|f(x, u ,d)\|_2 \\
  &\leq C_f \varepsilon,
\end{align*}
The chain of inequalities above utilizes several standard tricks to interchange the absolute value with the supremum over control inputs $u$ (resp. disturbances $d$) before applying the Cauchy--Schwarz inequality and the fact that $\|p(t, x)\|_2 < 1$ for all $(t, x) \in \Omega_K$. Finally, the last inequality utilizes the assumption that the dynamics are bounded.
\end{proof}
Because the bounds obtained in Lemma~\ref{lemma:unif-approx-bounds} are uniform in $(t, x) \in \Omega_K$, we can in fact write that
\begin{align*}
\|H_\theta - H\|_{\infty} < C_f \varepsilon, \quad \|V_\theta(T, \cdot) - \ell\|_\infty < \varepsilon,
\end{align*}
where $\|\cdot\|_\infty$ denotes the supremum norm on $\Omega_K$ as in the definition of the DeepReach loss in Equation~\eqref{eq:nn-loss}. With both of the above lemmas, we can now prove Theorem~\ref{th:existence}. 
\begin{proof}[Proof of Theorem~\ref{th:existence}]
By Lemma~\ref{lemma:univ-approx-est}, we obtain the existence of $V_\theta \in \mathfrak{C}_{n+1}(\sigma)$ such that
\begin{align*}
\varepsilon > \sup_{(t,x) \in \Omega_K} |V(t,x) - V_\theta(t, x)| &+ \sup_{(t,x) \in \Omega_K} |\partial_t V(t,x) - \partial_t V_\theta(t, x)| \\
&+ \sup_{(t, x) \in \Omega_K} \|\nabla_x V(t,x) - \nabla_x V_\theta(t, x)\|_2.
\end{align*}
It remains to show that $V_\theta$ satisfies the conditions in the statement of Theorem~\ref{th:existence}. Namely, we must show that $L(\theta) \leq C \varepsilon$ for some $C > 0$ that only depends on the dynamics of the underlying reachability problem. Recall that
\begin{align*}
  L(\theta) &= h_1(\theta) + \lambda h_2(\theta) \\
  &= \|V_\theta(T, x) - \ell(x)\|_\infty + \lambda \|\min\{\partial_t V_\theta(t, x) + H_\theta(t,x), \ell(x) - V_\theta(t, x)\}\|_{\infty},
\end{align*}
where $\lambda > 0$ is a tunable parameter. Now, Lemma~\ref{lemma:unif-approx-bounds} allows us to establish two useful lower bounds. First, because $V$ satisfies the HJI equation in Equation~\eqref{eq:hji-reach}, we have that
\begin{align*}
  \partial_t V(t, x) + H(t, x) \geq \min\{\partial_t V(t, x) + H(t, x), \ell(x) - V(t, x)\} = 0
\end{align*}
for all $(t, x) \in \Omega_K$. Thus, it follows that
\begin{align*}
  \partial_t V_\theta(t, x) + H_\theta(t, x) \geq \partial_t V_\theta(t, x) + H_\theta(t, x) - \partial_t V(t, x) + H(t, x) > - \varepsilon - C_f \varepsilon,
\end{align*}
applying two of the bounds from Lemma~\ref{lemma:unif-approx-bounds}. Taking $C := 2\max\{C_f, 1\}$, it follows that 
\begin{align}\label{eq:lower-bound-deriv}
  \partial_t V_\theta(t, x) + H_\theta(t, x) > -C \varepsilon
\end{align}
for all $(t, x) \in \Omega_K$. Similarly, observe that because 
\begin{align*}
  \ell(x) - V(t, x) \geq \{\partial_t V(t, x) + H(t, x), \ell(x) - V(t, x)\} = 0
\end{align*}
for all $(t, x)$, we have that
\begin{align}\label{eq:lower-bound-terminal}
  \ell(x) - V_\theta(t, x) \geq \ell(x) - V_\theta(t, x) - \ell(x) - V(t, x) = V_\theta(t, x) - V(t,x) > -\varepsilon \geq -C \varepsilon
\end{align}
again applying the corresponding bound from Lemma~\ref{lemma:unif-approx-bounds}. Combining Equation~\eqref{eq:lower-bound-deriv} and Equation~\eqref{eq:lower-bound-terminal}, it follows that
\begin{align*}
  \min\{\partial_t V_\theta(t, x) + H_\theta(t, x), \ell(x) - V_\theta(t, x)\} > -C \varepsilon
\end{align*}
for all $(t, x) \in \Omega_K$.

Now, because $V$ solves Equation~\eqref{eq:hji-reach}, for all $(t, x) \in \Omega_K$, we must have that either $\partial_t V(t, x) + H(t, x) = 0$ or $\ell(x) - V(t, x) = 0$ (or, both conditions could possibly hold, but this case is not relevant below). In the former case, we have that
\begin{align*}
  -C \varepsilon < \min\{\partial_t V_\theta(t, x) + H_\theta(t, x), \ell(x) - V_\theta(t, x)\} &\leq \partial_t V_\theta(t, x) + H_\theta(t, x) \\
  &= \partial_t V_\theta(t, x) + H_\theta(t, x) - \partial_t(t, x) - H(t, x) \\
  &\leq C \varepsilon,
\end{align*}
again applying the result of Lemma~\ref{lemma:unif-approx-bounds} in the last line. On the other hand, if $(t, x) \in \Omega_K$ is such that $\ell(x) - V(t, x) = 0$, then we see that 
\begin{align*}
  -C \varepsilon < \min\{\partial_t V_\theta(t, x) + H_\theta(t, x), \ell(x) - V_\theta(t, x)\} &\leq \ell(x) - V_\theta(t, x) \\
  &= \ell(x) - V_\theta(t, x) - (\ell(x) - V(t, x)) \\
  &= V(t, x) - V_\theta(t, x) \\
  &< \varepsilon \\
  &\leq C \varepsilon,
\end{align*}
applying the result of Lemma~\ref{lemma:unif-approx-bounds} yet again. Because the above two cases are exhaustive, it follows that
\begin{align*}
  |\min\{\partial_t V_\theta(t, x) + H_\theta(t, x), \ell(x) - V_\theta(t, x)\}| < C \varepsilon
\end{align*}
for all $(t, x) \in \Omega_K$. In other words,
\begin{align*}
  h_2(\theta) = \|\min\{\partial_t V_\theta(t, x) + H_\theta(t,x), \ell(x) - V_\theta(t, x)\}\|_{\infty} < C \varepsilon.
\end{align*}
Now, the fact that $\ell(x) = V(T, x)$ for all $x \in K$ implies that
\begin{align*}
  h_1(\theta) = \|V_\theta(T, x) - \ell(x)\|_\infty = \|V_\theta(T, x) - V(t, x)\|_\infty < \varepsilon \leq C \varepsilon,
\end{align*}
by Lemma~\ref{lemma:unif-approx-bounds}. Thus, it follows that
\begin{align*}
  L(\theta) = h_1(\theta) + \lambda h_2(\theta) < (C + \lambda C) \varepsilon.
\end{align*}
By taking $C' := C + \lambda C$, which only depends on the parameter $\lambda > 0$ and the dynamics $f$ of the underlying reachability problem, we conclude that for any $\varepsilon > 0$, there exists some set of parameters $\theta$ such that $L(\theta) < C' \varepsilon$.
\end{proof}
The above result ensures that the DeepReach loss is an appropriate metric for solving Equation~\eqref{eq:hji-reach} numerically: good approximations of the true solution to Equation~\eqref{eq:hji-reach} correspond to small loss values in Equation~\eqref{eq:nn-loss}. However, it is not particularly practical. To provide a practical guarantee, we turn to Theorem~\ref{th:convergence}. Rather than simply establishing the \textit{existence} of a network that makes the DeepReach loss arbitrarily small, Theorem~\ref{th:convergence} shows that by performing training under which the DeepReach loss converges to zero (which is possible by the above result), the corresponding neural networks must converge uniformly to the solution to Equation~\eqref{eq:hji-reach}. In some sense, this result justifies the use of the DeepReach algorithm, as it guarantees that training via gradient descent to minimize the DeepReach loss will result in estimators that uniformly converge to the true value function for a given reachability problem.

Before presenting the proof of Theorem~\ref{th:convergence}, several tools from PDE theory, and specifically, the theory of viscosity solutions to PDEs, are necessary. Note that, as discussed in \cite{reachability-overview} and \cite{evans-hji}, the HJI equation in Equation~\eqref{eq:hji-reach} admits a viscosity solution $V$. We utilize the following standard definition of viscosity solutions, introduced in \cite{evans-hji}:
\begin{definition}
\label{def:visc-solutions}
A function $v \in \sC((0, T) \times \intr(K))$ is: 
\begin{enumerate}
\item[(i)] a viscosity subsolution of Equation~\eqref{eq:hji-reach} if for any test function $\varphi \in \sC^1((0, T) \times \intr(K))$, $\mathcal{L}[\varphi] \leq 0$ for every local maximum $(t_0, x_0) \in (0, T) \times \intr(K)$ of $v - \varphi$ on $(0, T) \times \intr(K)$.
\item[(ii)] a viscosity supersolution of Equation~\eqref{eq:hji-reach} if for any test function $\varphi \in \sC^1((0, T) \times \intr(K))$, $\mathcal{L}[\varphi] \geq 0$ for every local minmum $(t_0, x_0) \in (0, T) \times \intr(K)$ of $v - \varphi$ on $(0, T) \times \intr(K)$.
\item[(iii)] a viscosity solution of Equation~\eqref{eq:hji-reach} if $v$ is both a viscosity subsolution and viscosity supersolution.
\end{enumerate}
\end{definition}
More background on viscosity solutions, their motivation, and their many useful properties can be found in \cite{evans-hji} and \cite{visc-users}. Without going into unnecessary technical detail, we remark that any classical solution to Equation~\eqref{eq:hji-reach} is also a viscosity solution, a fact that we leverage below \cite{visc-users}. Additionally, viscosity solutions to Equation~\eqref{eq:hji-reach} satisfy a standard comparison principle, in the sense that if $u$ is a viscosity supersolution and $v$ a viscosity subsolution, then $v \leq u$ on $[0, T) \times K$ \cite{evans-hji,visc-users}.

We also require the notion of a \textit{proper} nonlinear PDE, in the context of viscosity solutions. This concept, also from \cite{visc-users} is crucial for establishing uniform convergence guarantees for equations that admit viscosity solutions.
\begin{definition}\label{def:proper}
Suppose that a nonlinear PDE is of the form $F(x, u, Du, D^2 u) = 0$, where $F : \R^n \times \R \times \R^n \times \mathcal{S}^n \to \R$, with $\mathcal{S}^n$ denoting the set of symmetric matrices. Then, $F$ is proper if, holding all other inputs fixed,
\begin{align*}
  F(x, r, p, X) \leq F(x, s, p, Y)
\end{align*}
for all $r \leq s$ and $X \preceq Y$, with the latter denoting the standard order on PSD matrices.
\end{definition}
From the above definition a straightforward reformulation of Equation~\eqref{eq:hji-reach} shows that the HJI equation in question is necessarily proper.
\begin{lemma}\label{lemma:proper}
There exists a map $F : \Omega_K \times \R \times \R^{n+1} \times \mathcal{S}^{n+1} \to \R$ such that $V$ solves Equation~\eqref{eq:hji-reach} if and only if 
\begin{align*}
  F(t, x, V(t,x), \nabla V (t, x), \nabla^2 V(t, x)) = 0,
\end{align*}
and the resulting equation is proper in the sense of Definition~\ref{def:proper}.
\end{lemma}
\begin{proof}
Any solution to Equation~\eqref{eq:hji-reach} satisfies
\begin{align*}
\min\left\{\partial_t V(t, x) + H(t, x), \ell(x) - V(t, x)\right\} = 0
\end{align*}
on $\Omega_K$. Because $\min\{a, b\} = -\max\{-a, -b\}$, the above equality holds if and only if
\begin{align}\label{eq:max-reform}
  -\max\left\{-\partial_t V(t, x) - H(t, x), V(t, x) - \ell(x)\right\} = 0.
\end{align}
Thus, we take 
\begin{align}\label{eq:proper-op-hji}
  F(t, x, r, p, X) := \max\left\{-p_t - \sup_u \inf_d \langle p_x, f(x, u, d) \rangle, r - \ell(x)\right\},
\end{align}
where $p = (p_t, p_x) \in \R \times \R^n$, taking into account both the time and spatial derivatives of solutions to Equation~\eqref{eq:hji-reach}. Then, Equation~\eqref{eq:max-reform} holds if and only if 
\begin{align*}
  F(t, x, V(t,x), \nabla V (t, x), \nabla^2 V(t, x)) = 0.
\end{align*}
Furthermore, $F$ defines a proper equation because it does not depend on the input $X \in \mathcal{S}^{n+1}$, and if $r \leq s$, then $r - \ell(x) \leq s - \ell(x)$ for any $x \in \Omega_K$. In turn,
\begin{align*}
  F(t, x, r, p, X) &= \max\left\{-p_t - \sup_u \inf_d \langle p_x, f(x, u, d) \rangle, r - \ell(x)\right\} \\
  &\leq \max\left\{-p_t - \sup_u \inf_d \langle p_x, f(x, u, d) \rangle, s - \ell(x)\right\} \\
  &= F(t, x, s, p, X),
\end{align*}
holding all other inputs fixed. Thus, $F$ is proper.
\end{proof}
Before proving Theorem~\ref{th:convergence}, consider the following \textit{perturbed} version of Equation~\eqref{eq:hji-reach}, indexed by $k \in \N$.
\begin{align}
\label{eq:hji-perturbed}
\begin{split}
&\min\left\{\partial_t U(t, x) + H(t, x), \ell(x) - U(t, x)\right\} = \varepsilon^{(k)}(t, x), \\
& U(T, x) = \ell(x) + \delta^{(k)}(x).
\end{split}
\end{align}
From Theorem~\ref{th:existence}, it immediately follows that there exists a sequence of neural networks $\{V_{\theta^{(k)}}\}_{k \in \N}$ that solve perturbed HJI equations of the above form. 
Indeed, taking
\begin{align*}
  \varepsilon^{(k)}(t, x) := \min\left\{\partial_t V_{\theta^{(k)}}(t, x) + H_{\theta^{(k)}}(t, x), \ell(x) - U(t, x)\right\}, \quad \delta^{(k)}(x) := V_{\theta^{(k)}}(T, x) - \ell(x),
\end{align*}
Theorem~\ref{th:existence} establishes that the sequence $\{V_{\theta^{(k)}}\}_{k \in \N}$ satisfies Equation~\eqref{eq:hji-perturbed}, and we further have that
\begin{align*}
  L(\theta^{(k)}) = h_1(\theta^{(k)}) + \lambda h_2(\theta^{(k)}) = \|\varepsilon^{(k)}\|_\infty + \lambda \|\delta^{(k)}\|_\infty \to 0
\end{align*}
as $k \to \infty$. Conversely, it also holds that if $L(\theta^{(k)}) \to 0$, then the corresponding neural networks satisfy perturbed equations of he form in Equation~\eqref{eq:hji-perturbed}. Furthermore, by the construction of the DeepReach loss, it follows that $\|\varepsilon^{(k)}\|_\infty \to 0$ and $\|\delta^{(k)}\|_\infty \to 0$ as $k \to \infty$. From Theorem~\eqref{th:existence}, we simply know that such a sequence of parameters exists. 

Now, the sequence of above equations could instead be written as
\begin{align*}
  F(t, x, U(t, x), \nabla U(t, x), \nabla^2 U(t, x)) = \max\left\{-\partial_t U(t, x) - H(t, x), \ell(x) - U(t, x)\right\} = -\varepsilon^{(k)}(t, x),
\end{align*}
where $F$ is as in Lemma~\ref{lemma:proper}. In turn, defining
\begin{align}\label{eq:proper-op-hji-perturbed}
  F^{(k)}(t, x, U(t, x), \nabla U(t, x), \nabla^2 U(t, x)) := F(t, x, U(t, x), \nabla U(t, x), \nabla^2 U(t, x)) + \varepsilon^{(k)}(t, x),
\end{align}
it immediately follows from Lemma~\ref{lemma:proper} that the sequence in Equation~\eqref{eq:hji-perturbed} is proper, given by
\begin{align}\label{eq:proper-perturbed-equations}
  F^{(k)}(t, x, U(t, x), \nabla U(t, x), \nabla^2 U(t, x)) = 0,
\end{align}
and that $F^{(k)} \to F$ uniformly as $k \to \infty$.

The novelty of Theorem~\ref{th:convergence}, however, lies in establishing that for \textit{any} such sequence of neural networks $\{V_{\theta^{(k)}}\}_{k \in \N}$ obtained via training, $V_{\theta^{(k)}} \to V$ uniformly on $\Omega_K$. In essence, rather than simply establishing the existence of such a sequence, the following proof guarantees that standard training procedures such as gradient descent, coupled with standard techniques to avoid local minima (e.g., momentum, randomized restarts or perturbations, etc.), on the DeepReach loss will result in approximations that uniformly converge to the true solution to Equation~\eqref{eq:hji-reach} as training progresses. We require one more technical lemma, based on the convergence properties of viscosity solutions from \cite{visc-users}.
\begin{lemma}\label{lemma:visc-super-sub-conv}
Suppose that $u_k$ is a sequence of viscosity subsolutions to a sequence of proper equations, defined by $F^{(k)}(x, u, Du, D^2 u) = 0$, for $x \in K \subset \R^n$ compact. Then, 
\begin{align*}
  \overline{u}(x) := \lim_{j \to \infty} \sup\left\{u_k(y) : k \geq j, \; y \in K, \; \|y - x\|_2 \leq \frac{1}{j} \right\}
\end{align*}
is a viscosity subsolution to the equation $\overline{F}(x, u, Du, D^2 u) = 0$, where
\begin{align*}
  \overline{F}(x, u, Du, D^2 u) := \limsup_{k \to \infty} F^{(k)}(x, u, Du, D^2 u).
\end{align*}
Similarly, if $u_k$ is instead a sequence of viscosity supersolutions to $F_k(x, u, Du, D^2 u) = 0$, then 
\begin{align*}
  \underline{u}(x) := \lim_{j \to \infty} \inf\left\{u_k(y) : k \geq j, \; y \in K, \; \|y - x\|_2 \leq \frac{1}{j} \right\}
\end{align*}
is a viscosity supersolution to the equation $\underline{F}(x, u, Du, D^2 u) = 0$, where $\underline{F}$ is defined by taking the limit infimum of $F^{(k)}$ above.
\end{lemma}
\begin{proof}
This is the result of Lemma 6.1 and the corresponding Remarks 6.2, 6.3, and 6.4 in \cite{visc-users}.
\end{proof}
Equipped with the above technical lemma, we are finally prepared to prove Theorem~\ref{th:convergence}.
\begin{proof}[Proof of Theorem~\ref{th:convergence}]
Consider any sequence of parameters $\{\theta^{(k)}\}_{k \in \N}$ such that $L(\theta^{(k)}) \to 0$ as $k \to \infty$. As discussed above, such a sequence of parameters defines a family $\{V_{\theta^{(k)}}\}_{k \in \N}$ that satisfies the sequence of perturbed PDEs in Equation~\eqref{eq:hji-perturbed}, with error terms $\varepsilon^{(k)} : [0, T] \times K \to \R$ and $\delta^{(k)} : K \to \R$ satisfying
\begin{align*}
  \|\varepsilon^{(k)}\|_{\infty} \to 0, \quad \|\delta^{(k)}\|_\infty \to 0
\end{align*}
as $k \to \infty$. By Lemma~\ref{lemma:proper}, these perturbed PDEs are equivalently described by proper formulations, defined in Equation~\eqref{eq:proper-perturbed-equations}. Now, to invoke Lemma~\ref{lemma:visc-super-sub-conv}, we define
\begin{align*}
  \overline{V}(t, x) := \lim_{j \to \infty} \sup\left\{V_{\theta^{(k)}}(s, y) : k \geq j, \; (s, y) \in \Omega_K, \; \|(t, x) - (s, y)\|_2 \leq \frac{1}{j} \right\}
\end{align*}
and
\begin{align*}
  \underline{V}(t, x) := \lim_{j \to \infty} \inf\left\{V_{\theta^{(k)}}(s, y) : k \geq j, \; (s, y) \in \Omega_K, \; \|(t, x) - (s, y)\|_2 \leq \frac{1}{j} \right\}.
\end{align*}
Because $F^{(k)} \to F$ uniformly on $\Omega_K \times \R \times \R^{n+1} \times \mathcal{S}^{n+1}$ as $k \to \infty$, it follows that
\begin{align*}
  \liminf_{k \to \infty} F^{(k)} = F = \limsup_{k \to \infty} F^{(k)}.
\end{align*}
Thus, by Lemma~\ref{lemma:visc-super-sub-conv}, $\overline{V}(t, x)$ is a viscosity subsolution to the proper equation defined by $F$ in Equation~\eqref{eq:proper-op-hji}. Similarly, $\underline{V}(t, x)$ is a viscosity supersolution to the proper equation defined in Equation~\eqref{eq:proper-op-hji}. Note that, by construction, $\overline{V}$ and $\underline{V}$ satisfy the pointwise inequality
\begin{align*}
  \underline{V}(t, x) \leq \overline{V}(t, x)
\end{align*}
for all $(t, x) \in \Omega_K$. However, because $\overline{V}$ is a viscosity supersolution and $\underline{V}$ a viscosity subsolution, by the comparison principle for viscosity solutions in \cite[Thoerem 3.3]{visc-users}, we have that
\begin{align*}
  \overline{V}(t, x) \leq \underline{V}(t, x)
\end{align*}
for all $(t, x) \in \Omega_K$. Consequently, it follows that $\overline{V}(t, x) = \underline{V}(t, x)$, and the resulting function is both a viscosity subsolution and viscosity supersolution. Hence, by Definition~\ref{def:visc-solutions}, $\overline{V}(t, x) = \underline{V}(t, x)$ is a viscosity solution to Equation~\eqref{eq:proper-op-hji}. By Assumption~\hyperref[a:assump-B]{(B)}, however, Equation~\eqref{eq:proper-op-hji} admits a unique viscosity solution, so 
\begin{align*}
  \overline{V}(t, x) = \underline{V}(t, x) = V(t, x)
\end{align*}
Then, by \cite[Remark 6.4]{visc-users}, the construction of $\overline{V}$ and $\underline{V}$ respectively ensures that 
\begin{align*}
  \lim_{k \to \infty} V_{\theta^{(k)}}(t, x) = V(t, x)
\end{align*}
on any compact subset of $[0, T) \times K$. If not, then there would exists $\nu > 0$ and a sequence $\{n_k\}_{k \in \N}$ and $\{(t_k, x_k)\}_{k \in \N} \in [0, T) \times K$, with $(t_k, x_k) \to [0, T) \times K$, such that
\begin{align*}
  V_{\theta^{(n_k)}}(t_k, x_k) - V(t_k, x_k) > \varepsilon \quad \text{or} \quad V_{\theta^{(n_k)}}(t_k, x_k) - V(t_k, x_k) < -\varepsilon.
\end{align*}
Taking $k \to \infty$ however, and using the definitions of $\overline{V}$ and $\underline{V}$ respectively, the continuity of $V$ would imply that
\begin{align*}
  |V(t, x) - V(t, x)| \geq \varepsilon,
\end{align*}
a clear contradiction. Thus, uniform convergence holds any any compact subset of $[0, T) \times K$. To extend convergence to \textit{all} of $\Omega_K = [0, T] \times K$, recall that $\delta^{(k)}(x) = V_{\theta^{(k)}}(T, x) - \ell(x)$ converges uniformly to zero as $k \to \infty$, for all $x \in K$. As a result, it follows that 
\begin{align*}
  \lim_{k \to \infty} \sup_{(t, x) \in \Omega_K} |V_{\theta^{(k)}}(t, x) - V(t, x)| = 0
\end{align*}
as claimed.
\end{proof}
\end{document}